\documentclass[preprint,12pt]{elsarticle}

\usepackage{amssymb}
\usepackage{amsmath}
\usepackage{amsthm}
\usepackage{enumerate}
\usepackage{endnotes}

\newcommand{\old}[1]{{}}

\newcommand{\bb}{\mathbb}

\newcommand{\R}{\bb R}

\newcommand{\N}{\bb N}

\DeclareMathOperator\supp{supp}

\def\Item$#1${\item $\displaystyle#1$
   \hfill\refstepcounter{equation}(\theequation)}

\theoremstyle{definition}
\newtheorem{prop}{Proposition}[section]
\newtheorem{theorem}[prop]{Theorem}
\newtheorem{lemma}[prop]{Lemma}

\newtheorem{claim}{Claim}
\newtheorem{corollary}[prop]{Corollary}

\newtheorem{remark}[prop]{Remark}
\newtheorem{question}{Question}
\newtheorem{example}[prop]{Example}

\newtheoremstyle{TheoremNum}
        {\topsep}{\topsep}              
        {\itshape}                      
        {}                              
        {\bfseries}                     
        {.}                             
        { }                             
        {\thmname{#1}\thmnote{ \bfseries #3}}
    \theoremstyle{TheoremNum}
    \newtheorem{theorem-pre}{Theorem}

\numberwithin{equation}{section}

\journal{Operations Research Letters}

\begin{document}

\begin{frontmatter}



\title{On the sufficiency of finite support duals in semi-infinite linear programming}


\author[jhu]{Amitabh Basu}
\ead{basu.amitabh@jhu.edu}
\author[booth]{Kipp Martin}
\ead{kmartin@chicagobooth.edu}
\author[booth]{Christopher Thomas Ryan}
\ead{chris.ryan@chicagobooth.edu}

\address[jhu]{The Johns Hopkins University} 
\address[booth]{University of Chicago, Booth School of Business}

\begin{abstract}
We consider semi-infinite linear programs with countably many constraints indexed by the natural numbers. 
When the constraint space is  the vector space of all real valued sequences, we show that the finite support (Haar) dual is equivalent to the algebraic Lagrangian dual of the linear program. This settles a question left open by Anderson and Nash~\cite{anderson-nash}. This result implies that if there is a duality gap between the primal linear program and its finite support dual, then this duality gap  cannot be closed by considering the larger space of dual variables that define the algebraic Lagrangian dual. However,  if the constraint space corresponds to certain subspaces of all real-valued sequences,  there may be a strictly positive duality gap with the finite support dual,  but  a  zero duality gap with the algebraic Lagrangian dual. 
\end{abstract}


\end{frontmatter}

\section{Introduction}

We begin with a brief review of notation and  basic definitions for semi-infinite linear programs. 
Let $Y$ be a vector space. The \emph{algebraic dual} of $Y$ is the set of linear functionals with domain $Y$ and is denoted by  $Y'$. Let $\psi \in Y'$. The evaluation of $\psi$ at $y$ is denoted by $ \langle y, \psi \rangle$; that is, $ \langle y, \psi \rangle = \psi(y)$. We emphasize that the theory presented here deals with algebraic dual spaces and \emph{not} topological dual spaces. Discussion of how our work relates to topological duals appears in Remarks~\ref{remark:algebraic-dual-big-cone-small}~and~\ref{remark:contrast-with-16-3}.

Let $P$ be a convex cone in $Y$. A convex cone $P$ is {\em pointed} if and only if $P \cap -P = \left\{0\right\}$.  In the rest of the paper all convex cones are assumed to be pointed.  A pointed convex cone $P$ in $Y$ defines a vector space ordering $\succeq_P$ of $Y$, with $y \succeq_P y'$ if  $y - y' \in P$. 
The \emph{algebraic dual cone} of $P$ is
\begin{align*}
P' = \left\{\psi \in Y' : \langle y, \psi \rangle \ge 0 \text{ for all } y \in P\right\}.
\end{align*}
Elements of $P'$ are called \emph{positive linear functionals} on $Y$ (see for instance, page 17 of \cite{holmes}). Let $A:X \rightarrow Y$ be a linear mapping from vector space $X$ to vector space  $Y$. The \emph{algebraic adjoint} $A' : Y' \to X'$ is a linear operator defined by $A'(\psi) = \psi \circ A$ and satisfies $ \langle x, A'(\psi) \rangle = \langle A(x), \psi \rangle$ where $\psi \in Y'$ and $x \in X$.   Using this notation, define the primal  conic optimization problem
\begin{align*}\label{eq:conlp-primal}
\begin{array}{rl}
\quad \inf_{x\in X} & \langle x, \phi  \rangle \\
\textrm{s.t.} & A(x) \succeq_P b
\end{array}\tag{\text{ConLP}}
\end{align*}
where $b\in Y$  and $\phi$ is a linear functional on $X$.


Now  define the standard algebraic Lagrangian dual for~\eqref{eq:conlp-primal}.
{\small
\begin{eqnarray*}
\sup_{\psi \in P'} \inf_{x \in X}  \{  \langle x, \phi  \rangle  + \langle b - A(x), \psi \rangle \}  
& = & \sup_{\psi \in P'} \inf_{x \in X} \big\{  \langle x, \phi  \rangle  + \langle b, \psi \rangle -\langle  A(x), \psi \rangle \big\}  \\ 
& = & \sup_{\psi \in P'} \big\{  \langle b, \psi \rangle  + \inf_{x \in X} \{   \langle x, \phi  \rangle  -\langle  A(x), \psi \rangle \} \big\} \\\  
& = & \sup_{\psi \in P'} \big\{  \langle b, \psi \rangle  + \inf_{x \in X} \{  \langle x, \phi  \rangle -\langle x, A^{\prime}(\psi) \rangle \}  \big\} \\   
& = & \sup_{\psi \in P'} \big\{  \langle b, \psi \rangle  + \inf_{x \in X} \langle x, \phi - A'(\psi) \rangle \big\}. 
\end{eqnarray*}} 
Since $x \in X$ is unrestricted, if  $ \phi - A^{\prime}(\psi)$ is not the zero linear functional on $X,$ then the inner minimization goes to negative infinity, so  require $\phi - A^{\prime}(\psi) = \theta_{X}$, where $\theta_{X}$ is the zero linear functional on $X$. Then the Lagrangian dual of \eqref{eq:conlp-primal} is
\begin{align*}\label{eq:conlp-dual}
\begin{array}{rrl}
\sup & \langle  b, \psi \rangle & \\
  {\rm s.t.} & A^{\prime} (\psi) &= \phi \\
& \psi &\in P'.
\end{array}\tag{\text{ConDLP}}
\end{align*}
This problem is called the {\it algebraic} Lagrangian dual  of \eqref{eq:conlp-primal} since the linear functionals $\psi$ that define the dual problem are in $Y^{\prime}$, which is the algebraic dual of $Y.$

\paragraph{Semi-infinite linear programs}\label{s:silp}
Consider the case where  $X = \R^{n}$ and  $Y = \R^I$, i.e., the vector space of real-valued functions with domain $I$  where  $I$ is an arbitrary (potentially infinite) set. Let $a^1, a^2, \ldots, a^n$ and $b$ be functions in $Y = \R^I$. Let $A : \R^n \rightarrow \R^I$ be the linear mapping $x \mapsto (a^1(i)x_1 + a^2(i)x_2 + \ldots + a^n(i)x_n : i \in I)$. Let $\R_+^I$  denote the pointed cone of $u \in \R^I$ such that $u(i)\ge 0$ for all $i\in I$ and let $P = \R^I_+$.  With this specification for the vector spaces $X$ and $Y$, the map $A$, right hand side $b$ and cone $P$, problem~\eqref{eq:conlp-primal}  reduces to the standard semi-infinite linear program
\begin{align*}\label{eq:SILP}
\begin{array}{rl}
\qquad  \inf_{x\in \R^n} & \phi^\top x \\
 \textrm{s.t.} & \sum_{k = 1}^n a^k(i)x_k \geq b(i) \text{ for all } i\in I.
\end{array}\tag{\text{SILP}}
\end{align*}
 There is a slight abuse of notation here.   When  $X = \R^{n}$,  the algebraic dual $X^{\prime}$   is isomorphic to  $\R^{n}$   so each linear functional   $\phi \in X^{\prime}$  can be mapped to a  vector in $\R^n$. Thus, the primal objective function $\langle x, \phi \rangle$  in~\eqref{eq:conlp-primal}, is replaced by the inner product  $\phi^\top x$ with $\phi$ now treated as a vector in $\R^n.$ 

Next consider two alternative duals of \eqref{eq:SILP}: the algebraic Lagrangian dual and the finite support dual due to Haar \cite{haar1924}. \old{  Let $\ge$ denote the standard vector space ordering on $\R^I$. That is, $u \ge v$ if and only if $u(i) \ge v(i)$ for all $i \in I$.  }Recall that $(\R^I_+)'$ denotes the algebraic dual cone of $P=\R^I_+$. 
The algebraic Lagrangian dual of \eqref{eq:SILP} using~\eqref{eq:conlp-dual} is
\begin{align*}\label{eq:DSILP}
\begin{array}{rrl}
\sup & \langle  b, \psi \rangle & \\
  {\rm s.t.} & A^{\prime} (\psi) &= \phi \\
& \psi &\in (\R^I_+)'.
\end{array}\tag{\text{DSILP}}
\end{align*}

A second dual is  derived as follows. Instead of considering every linear functional $\psi \in (\R^I_+)'$ as above,  consider a subset of these linear functionals, called the {\em finite support elements}. For $u \in \R^I$, the \emph{support} of $u$ is the set $\supp(u) = \left\{i : u(i) \neq 0\right\}$. The subspace $\R^{(I)}$ denotes those functions in $\R^I$ with finite support.\old{ The subspace $\R^{(I)}$ inherits the  ordering defined earlier for $\R^{I}$.} Let $\R_+^{(I)}$  denote the pointed cone of $v \in \R^{(I)}$ such that $v(i)\ge 0$ for all $i\in I$\old{ and let  $(\R^{(I)}_+)'$ denote its dual cone}. 
Under the natural embedding of $\R^{(I)}$ into $(\R^I)'$ for $u \in \R^I$ and $v \in \R^{(I)}$, write $\langle u, v \rangle = \sum_{i \in I} u(i)v(i)$. The latter sum is well-defined since $v$ has finite support. Under this embedding, $\R^{(I)}_+$ is a subset of $(\R^I_+)'$. Moreover, under this embedding, $A' : (\R^I)' \to X' (=\R^n)$ restricted to $\R^{(I)}$ becomes the map $A'(v) = ( \sum_{i \in I} a^{k}(i) v(i) )_{k=1}^n$.   
The {\it finite support dual} is
\begin{align*}\label{eq:FDSILP}
\begin{array}{rrll}
\sup &\sum_{i \in I} b(i) v(i) & & \\
   {\rm s.t.} & \sum_{i \in I} a^{k}(i) v(i) &= \phi_k, & k = 1, \ldots, n \\
& v &\in \R_+^{(I)}. &
\end{array}\tag{\text{FDSILP}}
\end{align*}

The finite support dual  \eqref{eq:FDSILP}   is restricted    to the linear functionals   $\psi$ that can be   mapped to  $v \in \R_+^{(I)}$ under the standard embedding of $\R^{(I)}$ into $(\R^I)^{'}$.  Therefore $v(\ref{eq:FDSILP}) \le v(\ref{eq:DSILP})$  where the optimal value of optimization problem ($*$) is denoted by  $v(*)$. This leads naturally to the following question. 

\begin{question}\label{q}
Is it possible that 
$v(\ref{eq:SILP}) = v(\ref{eq:DSILP})$ and yet $v(\ref{eq:SILP}) > v(\ref{eq:FDSILP})$? 
In other words, can there exist a duality gap between the primal and its finite support dual that is closed by considering the algebraic Lagrangian dual?
\end{question}

This question is significant for the study of semi-infinite linear programming for at least two reasons. First, most duality theory has been developed for the finite support dual \cite{charnes1963duality,goberna2007sensitivity,goberna1998linear,karney81,lopez2012stability,ochoa2012stability}. Moreover, the only other dual given significant attention in the literature is the ``continuous dual" (see for instance \cite{goberna2010post,goberna1996stability}) and this dual shares many of the same duality properties as the finite support dual. Indeed, as stated by Goberna in \cite{goberna2010post}: ``all known duality theorems guaranteeing the existence of a zero duality gap have the same hypotheses for both dual problems [the finite support dual and the continuous dual]". He even goes so far to say that the finite support dual and the continuous dual are ``equivalent in practice." 

Second, the algebraic Lagrangian dual is notoriously challenging to characterize and work with. Indeed, to the author's knowledge, little has been said about the algebraic dual in the semi-infinite programming literature (only a few studies mention it, and they do not draw conclusions about its connection with the finite support dual \cite{anderson-nash,nash1985algebraic}).
 
\old{Borwein answered Question~\ref{q} negatively in the case where $I$ is endowed with a topology that makes it into a compact Hausdorff space \cite{borwein1983adjoint}.} 

To the authors' knowledge Question~\ref{q} has not been settled for $I = \N$, i.e., semi-infinite linear programs with countably many constraints. Indeed, on page 66 of Anderson and Nash's seminal work \cite{anderson-nash} they write: ``It seems to be hard, if not impossible, to find examples of countable semi-infinite programs which have a duality gap in this formulation [the finite support dual], but have no duality gap when we take $W$ to be a wider class of sequences" where $W$ refers to the vector space of dual variables. In  our notation,   $W = (\R^\N)'$  in  \eqref{eq:DSILP}  and  $W = \R^{(\N)}$   in \eqref{eq:FDSILP}.  
Semi-infinite linear programs with countably many constraints have been well-studied in the literature, particularly from the perspective of duality \cite{charnes1963duality, karney81,karney85}. In fact, one can even show that, theoretically, there is no loss in generality in considering the countable case. Theorem 2.3 in \cite{karney85} shows that every semi-infinite linear program with uncountable many constraints can be equivalently reposed over a countable subset of the original constraints.

The main result of this paper (Theorem~\ref{thm:no-gap}) proves that the answer to Question~\ref{q} is no for the case of $I = \N$, settling Anderson and Nash's open question. We show that $v(\ref{eq:DSILP}) = v(\ref{eq:FDSILP})$ by establishing that \eqref{eq:DSILP} and \eqref{eq:FDSILP} are equivalent programs.



However, there is a subtlety in Question~\ref{q} to keep in mind  for semi-infinite linear programs with countably many constraints. In the above discussion, a semi-infinite linear program with countably many constraints was cast as an instance of~\eqref{eq:conlp-primal} with $X = \R^n$, $Y=\R^{\N}$, $A : X \to Y$ defined by $A(x)= (a^1(i)x_1 + a^2(i)x_2 + \ldots + a^n(i)x_n : i \in I)$, and $P = \R^\N_+$. Then~\eqref{eq:DSILP} was formed using~\eqref{eq:conlp-dual}. However, if the functions $a^1, a^2, \ldots, a^n$ and $b$ lie in a subspace $\mathcal{V} \subseteq \R^\N$, then we may use $Y = \mathcal{V}$ and $P = \mathcal{V}\cap \R^\N_+$ to write the semi-infinite linear program as an instance of~\eqref{eq:conlp-primal}. The corresponding~\eqref{eq:conlp-dual} is
 \begin{align*}\label{eq:DSILPprime}
\begin{array}{rrl}
\sup & \langle  b, \psi \rangle & \\
  {\rm s.t.} & A^{\prime} (\psi) &= \phi \\
& \psi &\in (\mathcal{V}\cap \R^\N_+)'
\end{array}\tag{\text{DSILP($\mathcal{V}$)}}
\end{align*}where $(\mathcal{V}\cap \R^\N_+)' \subseteq \mathcal{V}'$ is the dual cone of $P=\mathcal{V}\cap \R^\N_+,$ which lies in the algebraic dual of $\mathcal{V}$.


It is quite possible that a positive linear functional defined on $(\mathcal{V}\cap \R^\N_+)'$ cannot be extended to $(\R^N_+)'$. This implies \eqref{eq:DSILP} (with $I=\N$) may have a smaller value than \eqref{eq:DSILPprime}, i.e., $v(\ref{eq:DSILP}) < v(\ref{eq:DSILPprime})$.  In this context, the following question is a natural extension of Question~\ref{q}.

\begin{question}\label{q2}
Is it possible that $v(\ref{eq:SILP}) = v(\ref{eq:DSILPprime})$ and  $v(\ref{eq:SILP}) > v(\ref{eq:FDSILP}) = v(\ref{eq:DSILP})$ when $a^1, \ldots, a^n, b \in \mathcal{V}$ for some subspace $\mathcal{V} \subseteq \R^\N$? In other words, when the constraint space $\mathcal V$ lies in a subspace of $\R^\N$, can there exist a duality gap between the primal and its finite support dual \eqref{eq:FDSILP}, that is closed by considering the algebraic Lagrangian dual defined according to that subspace?
\end{question}

We show in Section~\ref{s:duality-gaps} that this can happen.  More concretely,  in Example~\ref{example:not-primal-optimal}  in Section~\ref{s:duality-gaps}, there is  a duality gap between (SILP) and the finite support dual \eqref{eq:FDSILP}. However, if $a^1, \ldots, a^n, b$ are considered as elements of the space of convergent real sequences $c$,  
then~\eqref{eq:SILP}    is a special case of~\eqref{eq:conlp-primal} with $X=\R^n$, $ \mathcal V =Y=c$, $A : X \to Y$ and $P= c_+$ (the cone of convergent sequences with nonnegative entries), and there is no duality gap with its algebraic Lagrangian dual \eqref{eq:DSILc}
\begin{align*}\label{eq:DSILc}
\begin{array}{rrl}
\sup & \langle  b, \psi \rangle & \\
  {\rm s.t.} & A^{\prime} (\psi) &= \phi \\
& \psi &\in (c_+)'
\end{array}\tag{\text{DSILP($c$)}}
\end{align*}
where $(c_+)'$ is the algebraic dual cone of $c_+$. The same result also holds when $Y$ is the subspace of bounded real sequences $\ell_\infty$. 

\section{Main result}

\begin{lemma}\label{lemma:pos-functionals}
For every $\psi \in (\R^{\N}_+)'$, there exists $u \in \R_+^{(\N)}$ such that $\langle y, \psi \rangle = \sum_{i\in \N}y(i)u(i)$ for every $y \in \R^\N$. In other words, every positive linear functional on $\R^\N$ can be represented by a positive finite support dual vector.
\end{lemma}

\begin{proof}
Consider any $\psi \in (\R_+^\N)^\prime$, i.e.,  $\langle y, \psi \rangle \geq 0$ for all $y \geq 0$.  We show that $\psi$ can be represented by a finite support linear functional.
\begin{claim}\label{claim:finite_M}
There exists $M \in \N$ such that for all $v \in \R^\N_+$ whose first $M$ components are zero, $\langle v, \psi \rangle = 0$.
\end{claim}
\begin{proof}
Suppose no such $M$ exists. Then for every $n \in \N,$ there exists $v^n \in \R^\N_+$ such that $\langle v^n, \psi \rangle > 0$ (we can assume strictly greater than zero without loss) and the first $n$ components of $v^n$ are zero. Consider the sequence of vectors $\hat v^n = \frac{v^n}{\langle v^n , \psi \rangle}$ for $n \in \N$. Observe that $\langle \hat v^n, \psi\rangle = 1$ for all $n \in N$.
 Now, consider vector $a \in \R^\N_+$ where $a(i) = \sum_{n \in \N} \hat v^n(i)$ = $\sum_{n = 1}^{i} \hat v^n(i) + \sum_{n > i} \hat v^n(i)$ for all $i \in \N$.  Since $\hat v^n(i) = 0$ for all $n > i$,    $\sum_{n > i} \hat v^n(i) = 0$. This implies  $\sum_{n \in \N} \hat v^n(i)$ is a finite sum and  therefore well-defined.   Given any $N \in \N$,   it follows  from the definition of $a(i)$,  that  $a(i) - \sum_{n=1}^{N} \hat v^n(i) = \sum_{n \in \N} \hat v^n(i) - \sum_{n=1}^{N} \hat v^n(i) = \sum_{n > N} \hat v^n(i)$.  We have  $\sum_{n > N} \hat v^n(i)  \geq 0$  since the $\hat v^n$ are in $\R^\N_+$.  This implies  $a - \sum_{n = 1}^N \hat v^n \geq 0$. 
 Thus, $\langle a - \sum_{n = 1}^N \hat v^n, \psi \rangle \geq 0$ for every $N \in \N$. By linearity of $\psi$, this implies that $\langle a, \psi \rangle \geq \sum_{n=1}^N \langle \hat v^n, \psi \rangle = N$ for every $N \in \N$. But this means $\langle a, \psi \rangle$ cannot be a finite number, which is a contradiction of the fact that $\psi$, being a linear functional, is real valued. 
\end{proof}

Let $M$ be the natural number from Claim~\ref{claim:finite_M}. Now let $e^i$ denote the element of $\R^\N$ with 1 in the $i$-th coordinate and $0$ everywhere else. Let $u \in \R^{(\N)}$ be a finite support element given as follows
$$
u(i) = \left\{ \begin{array}{rl} \langle e^i, \psi \rangle, & i \leq M \\ 0, & i > M.\end{array}\right. 
$$
\begin{claim}\label{claim:positive-elements}
$\langle y, \psi \rangle = \sum_{i \in \N}u(i) y(i)$ for every $y \in \R^\N_+$.
\end{claim}
\begin{proof} 
Observe that any $y \in \R^\N_+$ can be represented as $y = \sum_{i = 1}^M y(i)e^i + v$ where $v \in \R^N_+$ has zeros in its first $M$ components.  By Claim~\ref{claim:finite_M}, $\langle v, \psi \rangle = 0$. Hence $\langle y, \psi\rangle = \sum_{i=1}^M y(i) \langle e^i, \psi \rangle + \langle v, \psi \rangle = \sum_{i \in \N}u(i) y(i)$.\end{proof}

Claim~\ref{claim:positive-elements} only applies to $y \geq 0$, i.e., the nonnegative elements in $\R^\N$. To complete the proof, we need to show that for arbitrary $y \in \R^\N$, $\langle y, \psi \rangle = \sum_{i \in \N}u(i) y(i)$. Define $y^+, y^- \in \R^\N_+$ as follows: for each $i\in \N$, $y^+(i) = \max\{y(i), 0\}$ and $y^-(i) = \max\{-y(i),0\}$. Thus, we can write $y \in \R^\N$ as $y = y^+ - y^-$. Therefore,
$$
\begin{array}{rcl}
\langle y, \psi \rangle  & = &\langle y^+, \psi \rangle - \langle y^-, \psi \rangle \\
& = & \sum_{i \in \N}u(i) y^+(i) - \sum_{i \in \N}u(i) y^-(i) \\
& = & \sum_{i \in \N}u(i) (y^+(i) - y^-(i)) \\
& = & \sum_{i \in \N}u(i) y(i)
\end{array}
$$ where the second equality follows from Claim~\ref{claim:positive-elements}.\end{proof}

Recall every $u \in \R_+^{(\N)}$ maps to a positive linear functional $\psi$ over $\R^\N$ via $\psi(x) = \sum_{i=1}^\infty u(i)x(i)$. Thus, $\R^{(\N)}_+$ can be embedded into a subset of $(\R^\N_+)'$. Combined with Lemma~\ref{lemma:pos-functionals}, this implies $(\R^\N_+)' \cong \R_+^{(\N)}$. In other words, the algebraic dual cone of $\R^\N_+$ is isomorphic to the positive cone $\R_+^{(\N)}$ in $\R^{(\N)}$.

\begin{remark}\label{remark:algebraic-dual-big-cone-small}
The fact that $(\R^\N_+)'$ is isomorphic to $\R_+^{(\N)}$ does not contradict the well-known fact that the full algebraic dual $(\R^{\N})'$ of $\R^{\N}$ is difficult to characterize. Indeed, the full algebraic dual $(\R^\N)'$ is of uncountable dimension (see page 195 of \cite{hitchhiker}) whereas the algebraic dual cone $(\R^\N_+)'$ is isomorphic to a subset of $\R^{(\N)}$, which has countable dimension. One intuitive justification is that $\R^\N_+$ is not a ``full dimensional" subset of $\R^\N$. This is because $\R^\N_+$ has empty interior in every linear topology. For a justification of this fact see~\ref{ap:no-interior}. 
\end{remark}

\begin{remark}\label{remark:contrast-with-16-3}
Our Lemma~\ref{lemma:pos-functionals} shares some similarities with Theorem~16.3 of \cite{hitchhiker}. We emphasize that Lemma~\ref{lemma:pos-functionals} does not contradict, nor depend on,  Theorem 16.3 of \cite{hitchhiker}. The latter result states that, under the usual product topology of $\R^\N$, the \emph{topological dual} $(\R^\N)^*$ of $\R^\N$ is isomorphic to $\R^{(\N)}$. Lemma~\ref{lemma:pos-functionals} differs from this result in two ways. First, Lemma~\ref{lemma:pos-functionals} demonstrates an equivalence between the algebraic dual cone of $\R^\N_+$ and the positive cone of $\R^{(\N)}$. In other words, Theorem~16.3 of \cite{hitchhiker} concerns \emph{all continuous} linear functionals on $\R^\N$ whereas our result concerns only \emph{positive} linear functionals. Second, our result is a statement about algebraic duality whereas Theorem~16.3 of \cite{hitchhiker} concerns topological duality. Our proof of Lemma~\ref{lemma:pos-functionals} is direct and entirely algebraic. It does require Theorem~16.3 of \cite{hitchhiker} or any topological concepts. 
\end{remark}

Using Lemma~\ref{lemma:pos-functionals}, Question~\ref{q} is answered for $I = \N$.

\begin{theorem}\label{thm:no-gap}
When $I = \N$, $v(\ref{eq:SILP}) = v(\ref{eq:DSILP})$ if and only if $v(\ref{eq:SILP}) = v(\ref{eq:FDSILP})$. 
\end{theorem}
\begin{proof}
It  suffices to show that $v(\ref{eq:DSILP}) = v(\ref{eq:FDSILP})$. Lemma~\ref{lemma:pos-functionals} implies  $(\R_+^\N)' \cong \R_+^{(\N)}$ and thus the feasible regions of \eqref{eq:DSILP} and \eqref{eq:FDSILP} are equivalent under the standard embedding of $\R^{(\N)}$ into $(\R^\N)'$. The objectives are also equivalent under that embedding and thus $v(\ref{eq:DSILP}) = v(\ref{eq:FDSILP})$.
\end{proof}

\section{Duality gaps in proper subspaces of $\R^\N$}\label{s:duality-gaps}

By Theorem~\ref{thm:no-gap},  the optimal value of the  finite support dual  is equal to the optimal value of  the algebraic Lagrangian dual for semi-infinite linear programs with countably many constraints, when we model them as an instance of \eqref{eq:conlp-primal} using $Y = \R^\N$. However, this is not necessarily true for problems with countably many constraints  when they are modeled as \eqref{eq:conlp-primal} with $Y$ as a proper subspace of $\R^\N$. In this section, we give an affirmative answer to Question~\ref{q2} via Example~\ref{example:not-primal-optimal}.

Two examples of proper subspaces of $\R^\N$ are $\ell_\infty,$  the space of all bounded real sequences, and $c,$  the space of all convergent sequences. Clearly, $c \subset \ell_\infty \subset \R^\N$. We extend the notion of positive linear functionals to these subspaces. A linear functional $\psi \in X'$ on any subspace $X \subseteq \R^\N$, is called a positive linear functional on $X$ if $\langle v, \psi \rangle \geq 0$ for all $v \in \R^{\N}_+ \cap X$. 

Restricting to a subspace of $\R^\N$ allows for more positive linear functionals. Define the \emph{limit functional} $\psi$ on $c$ by 
\begin{align}\label{eq:limit-func}
\langle v, \psi \rangle = \lim_{i\to \infty} v (i).
\end{align}
Clearly, $\psi$ is a positive linear functional over $c$. The next result shows that $\psi$ cannot be extended to a positive linear functional over all of $\R^\N$.
\begin{lemma}\label{lem:limit-func}
The limit functional $\psi$ defined in \eqref{eq:limit-func} cannot be extended to a positive linear functional on $\R^\N$, i.e., it cannot be extended to an element of $(\R^{\N}_+)' \subseteq (\R^\N)'$.
\end{lemma} 

\begin{proof}
%
Assume $\overline{\psi}$ is an extension of $\psi$ and that $\overline{\psi}$ is  a positive linear functional on $\R^\N$. We shall derive a contradiction. Let $\mathbf{1} \in  \R^{\N}_+ \cap c$ be the all ones sequence. Then $\langle \mathbf{1}, \overline{\psi} \rangle = \langle \mathbf{1}, \psi\rangle = 1$. Let $v = (n)_{n \in \N}$ be the sequence $(1,2,3, \ldots)$. Since $v$ is in the nonnegative orthant $\langle v, \overline{\psi} \rangle = \alpha \geq 0$. Let $M = \lceil \alpha \rceil + 1$. Consider the sequence $m = (0,0,0, ..., M, M, M, \ldots)$ where the first $M-1$ entries are 0. Then $v - m$ is in the nonnegative orthant, but $\langle v-m, \overline{\psi} \rangle = \langle v, \overline{\psi} \rangle - \langle m, \overline{\psi} \rangle = \alpha - \langle m, \psi \rangle = \alpha - M < 0$, and this is a contradiction to the assumption that $\overline{\psi}$ is a positive linear functional. Therefore $\psi$ cannot be extended to a positive linear functional on $\R^\N$.
%
\end{proof}


Although positive linear functionals on the space $c$ cannot be extended to    $(\R^{\N}_+)'$, they can be extended to positive linear functionals on $\ell_\infty$ as shown in Lemma~\ref{lem:limit-func-extend} below. 
\old{We first need to introduce some notions. A vector space $X$ is called an {\em ordered vector space} if it has a partial order $\succeq$ satisyfing 
\begin{enumerate}
\item $x \succeq y$ implies $x + z \succeq y + z$ for all $x, y, z \in X$.
\item $x \succeq y$ implies $\alpha x \succeq \alpha y$ for all $\alpha \in \R_+$ and $x, y \in X$.
\end{enumerate}
An ordered vector space is called a {\em Riesz space} if it is also a lattice under the partial order. $\R^\N$, $\ell_\infty$ and $c$ are all Riesz spaces under the standard partial order $x \succeq y$ if and only if $x\geq y$. Suppose $X$ is a Riesz spce, and let $M$ be a subspace of $X$. We say that $M$ {\em majorizes} $X$ if for every $x \in X$, there exists $y \in M$ such that $y \geq x$.
}


First, recall the notion of a {\em core point}. Given a vector space $X$ and a subset $A \subseteq X$, a point $a \in A$ is called a core point of $A$ if for every $x \in X$, there exists $\epsilon > 0$ such that $a + \lambda x \in A$ for all $0 \leq \lambda \leq \epsilon$. Some authors call such a point an {\em internal point} (see for instance Definition 5.58 in \cite{hitchhiker}). The following is a useful result for extending positive linear functionals.

\begin{theorem}[Krein-Rutman theorem, see Holmes \cite{holmes} p. 20] \label{thm:krein-rutman}
Let $X$ be a vector space ordered by $\succeq_P$ where $P$ is a pointed, convex cone in $X$. Furthermore, let $M$ be a linear subspace of $X$ ordered by $\succeq_{P \cap M}$. If $P \cap M$ contains a core point (with respect to $X$) of $P$, then any positive linear functional on $M$ admits a positive linear extension to all of $X$. In other words, if $\psi : M \to \R$ satisfies $\langle m, \psi \rangle  \ge 0$ for all $m \in M \cap P$,  then there exists a $\overline{\psi} : X \to \R$ with $\langle x,  \overline{\psi} \rangle \ge 0$ for all $x \in P$ and $\langle x,  \overline{\psi} \rangle = \langle x, \psi \rangle$ for all $x \in M$. 
\end{theorem}

\begin{lemma}\label{lem:limit-func-extend}
Every positive linear functional on $c$ can be extended to a positive linear functional on $\ell_\infty$.
\end{lemma}
\begin{proof} 
 Let $P =  \R^{\N}_+  \cap \ell_\infty$ be the nonnegative cone in $\ell_\infty$.  Take the convergent sequence $\mathbf{1} = (1,1,\dots)$. This convergent sequence is an element of  $P \cap c$.  Also, $\mathbf{1}$ is a core point of $P$ with respect to $\ell_\infty$.  To see that  $\mathbf{1}$ is a core point of $P$, take any sequence $(a_n)_{n \in \N} \in \ell_\infty.$  
Since $(a_n)$ is in  $\ell_\infty$,     $\sup_n |a_n|  < \infty$  and $\mathbf{1}    +  \lambda (a_n) \in P$ for all $\lambda \in (0, 1/\sup_n |a_n|).$  Since $P$ has a core point  with respect to $\ell_\infty$,  apply the Krein-Rutman theorem to extend positive  linear functionals defined on $c$ to positive linear functionals defined on $\ell_\infty$.
\end{proof}

\old{Observe that $c$ majorizes $\ell_\infty$ : for any $x \in \ell_\infty$, let $y\in c$ be the sequence whose every entry is $\lVert x \rVert_\infty$. By Theorem 8.32 in~Aliprantis and Border~\cite{hitchhiker} we can extend every positive linear functional on $c$ to a linear functional on $\ell_\infty$.}

\begin{corollary}\label{cor:limit-func-extend}
The limit functional defined in \eqref{eq:limit-func} can be extended to a positive linear functional over $\ell_\infty$. 
\end{corollary}

\old{By Theorem~\ref{thm:no-gap},  if  $I = \N$ then  the  duality gap between (\ref{eq:SILP})  and  its  finite support dual~\eqref{eq:FDSILP} is equal to the duality gap between (\ref{eq:SILP})  and its algebraic Lagrangian  dual (\ref{eq:DSILP}).    However, if $I = \N$, and  the constraint space  of  (\ref{eq:SILP}) corresponds to a  subspace of $\R^\N,$  then the algebraic Lagrangian dual with respect to this subspace   can yield a smaller duality gap than the gap that results from the finite support dual.  In this case, using the finite support dual can create a duality gap.}

Example \ref{example:not-primal-optimal} below provides an affirmative answer to Question~\ref{q2}.  This example gives an  \eqref{eq:SILP} with $I=\N$, where $v(\ref{eq:SILP}) = v(\ref{eq:DSILPprime})$ and $v(\ref{eq:SILP}) > v(\ref{eq:FDSILP}) = v(\ref{eq:DSILP})$,  with the subspace $\mathcal{V} = c$ and $\mathcal{V} =\ell_\infty$.\old{This is illustrated next in Example~\ref{example:not-primal-optimal} where the constraint space is the subspace $c \subset \ell_{\infty} \subset \R^\N.$} The equality $v(\ref{eq:FDSILP}) = v(\ref{eq:DSILP})$  follows  from Theorem~\ref{thm:no-gap}.



\begin{example}\label{example:not-primal-optimal}
The  \eqref{eq:SILP} is
\begin{align}\label{eq:not-primal-optimal}
\begin{array}{rcl}
\inf x_{1} && \\
x_{1} + \tfrac{1}{i^{2}}x_{2} &\ge& \tfrac{2}{i}, \quad i \in \N. \\
\end{array}
\end{align}
We show that $(x_1, x_2) = (\frac{1}{\delta}, \delta)$ is feasible to \eqref{eq:not-primal-optimal} for all $\delta > 0$. For every $i \in \N$, 
$$\begin{array}{rcl}
&(i - \delta)^2 &\geq 0 \\
\Rightarrow & i^2 + \delta^2 - 2\delta i & \geq 0 \\
\Rightarrow & i^2 + \delta^2 & \geq 2\delta i \\
\Rightarrow & \frac{i^2 + \delta^2}{\delta i^2} &\geq  \frac{2\delta i}{\delta i^2} \\
\Rightarrow & \frac{1}{\delta} + \frac{\delta}{i^2} &\geq \frac{2}{i}.
\end{array} 
$$
Letting $\delta \to \infty$ gives feasible solutions whose objective values converge to an objective value of $0$. 

Next observe that if  $(\bar x_1, \bar x_2)$  is any feasible solution, then  $\overline{x}_{1} \ge 0.$ Indeed, taking $i \to \infty$ in \eqref{eq:not-primal-optimal} leaves $x_1 \ge 0$.
Therefore the  optimal primal objective value is  equal to $0$. 


The finite support dual \eqref{eq:FDSILP} for this semi-infinite linear program is infeasible.  The objective coefficient of $x_2$ is $0$ and the coefficient of $x_2$ is strictly positive in the constraints, and so the only possible dual element satisfying the dual constraint corresponding to $x_2$ is $u = (0,0,\dots, 0,\dots)$; however, the objective coefficient of $x_1$ is $1$ and the dual constraint corresponding to $x_1$ is not satisfied for $u=0$. This shows that using the finite support dual leads to an infinite duality gap.

In this example,  $a^{1}(i) = 1$ for all $i \in \N$ so this sequence converges to 1.   Also, $a^{2}(i) = \frac{1}{i^{2}}$ for all $i \in \N$  and this sequence converges to 0. \old{ Thus, for any $x \in \R^{2}$,  $A(x) \in c.$}  Likewise, $b(i) = \frac{2}{i}$ for all $i \in \N$ so this sequence converges to 0.  Therefore, this semi-infinite linear program is  an instance of \eqref{eq:conlp-primal} with $Y = c \textrm{ (or }\ell_{\infty}\textrm{)} \subset \R^\N.$ Now consider the algebraic Lagrangian dual~\eqref{eq:DSILPprime}    defined  on  the space of positive linear functionals on $c$ (or $\ell_\infty$).   The limit functional  $\psi$ defined in 
\eqref{eq:limit-func} is  $\langle v, \psi \rangle = \lim_{i\to \infty} v (i)$, for   all $v \in \R^{\N}_+ \cap c$.  
This  is a positive linear functional on $c$ (as is its extension to $\ell_\infty$ by Corollary~\ref{cor:limit-func-extend}) since any convergent sequence of nonnegative numbers has a nonnegative limit.   As observed earlier, $a^{1}(i)$ converges   to 1 which is  the coefficient of $x_1$ in the objective, and $a^{2}(i) $ converges to 0 which is  the coefficient of $x_2$ in the objective. Therefore $\psi$ is a feasible dual solution. The dual objective value  $\langle b, \psi   \rangle$  is $\lim_{i \in \N} b(i) = \lim_{i \in \N} \frac{2}{i} = 0$. This is the optimal value of the primal and therefore the duality gap is zero. 

This example illustrates 
Lemma~\ref{lem:limit-func}.  If $\psi$ could be extended to a  positive linear functional $\overline{\psi}$ on $\R^\N$, then $\overline{\psi}$ would be a feasible solution to~\eqref{eq:DSILP} with objective function value zero implying  a zero duality gap between~\eqref{eq:not-primal-optimal} and ~\eqref{eq:DSILP}.  Since there is an infinite duality gap between~\eqref{eq:not-primal-optimal}    and  its finite support dual~\eqref{eq:FDSILP}, this would contradict  Theorem~\ref{thm:no-gap}. \hfill$\triangleleft$
\end{example}

\old{
\begin{example}\label{example:not-primal-optimal}
Consider the following instance of \eqref{eq:SILP}:
\begin{align}\label{eq:not-primal-optimal}
\begin{array}{rcl}
\inf x_{1} && \\
x_{1} + \tfrac{1}{t^{2}}x_{2} &\ge& \tfrac{1}{t^{2}} + \tfrac{1}{t}, \quad t \ge 1 \\
          x_1 \phantom{ + \tfrac{1}{t^{2}}x_{2}} & \ge & 0
\end{array}
\end{align}
\old{Applying the Fourier-Motzkin procedure to
\begin{align}\label{eq:full-system-example}
\begin{array}{rcl}
-x_1 \phantom{ + \tfrac{1}{t^{2}} x_{2}} + z &\ge& 0 \\
\phantom{-}x_1 + \tfrac{1}{t^{2}} x_{2}  \phantom{ + z} &\ge& \tfrac{1}{t^{2}} + \tfrac{1}{t}, \quad t \ge 1 \\
x_1 \phantom{ + \tfrac{1}{t^{2}} x_{2} + z} & \ge & 0
\end{array}
\end{align}
yields (by eliminating $x_1$)
\begin{equation}\label{eq:project-out-x1}
\begin{array}{rcl}
\tfrac{1}{t^{2}} x_{2} + z  &\ge& \tfrac{1}{t^{2}} + \tfrac{1}{t}, \quad t \ge 1 \\ 
\phantom{\frac{1}{t^{2}} x_{2} +}  z & \ge & 0.
\end{array}
\end{equation}
The only $I_3$ constraint is $z \ge 0$ and so $\sup_{h \in I_3} \tilde{b}(h) = 0$. We claim that for $\delta \ge 1$
\begin{eqnarray*}
\omega(\delta) = \sup_{t \ge 1} \left\{ \tfrac{1}{t^{2}} + \tfrac{1}{t}  -   \tfrac{\delta}{t^{2}} \right\} = \tfrac{1}{\delta^{2}}. 
\end{eqnarray*}
Note $\frac{\delta}{t^2} \le \frac{t}{t^2} = \frac{1}{t}$ whenever $t \ge \delta$ and so $\frac{1}{t^2} + \frac{1}{t}- \frac{\delta}{t^2} \le \frac{1}{t^2} \le \frac{1}{\delta^2}$ whenever $t \ge \delta$. On the other hand, when $1 \le t < \delta$ note that $\frac{1}{t} > \frac{1}{\delta}$ and $\frac{1}{t^2} > \frac{1}{\delta^2}$ and so $\frac{1}{t^2} + \frac{1}{t} - \frac{\delta}{t^2} \le \frac{1}{\delta^2} + \frac{1}{\delta} - \frac{\delta}{\delta^2} = \frac{1}{\delta^2}$. Note further that taking $t = \delta$ (this is appropriate since we assumed $\delta \ge 1$) gives $\frac{1}{t^2} + \frac{1}{t} - \frac{\delta}{t^2} = \frac{1}{\delta^2}$. Thus, we can conclude $\omega(\delta) = \frac{1}{\delta^2}$ for $\delta \ge 1$.   
Clearly then, $\lim_{\delta \rightarrow \infty} \omega(\delta) =    0 = \sup_{h \in I_3} \tilde{b}(h)$ and so by Lemma~\ref{lemma:primal-optimal-value} the optimal value is $0$. 

However, for $z = 0$ the system \eqref{eq:project-out-x1} has no possible feasible assignment for $x_2$. Indeed, for any proposed $\bar x_2$ take $t \ge \bar x_2$ then $ \tfrac{1}{t^{2}} \bar x_{2} + 0  \le \tfrac{1}{t} < \tfrac{1}{t^2} + \frac{1}{t}$, which implies $(\bar x_2, 0)$ is infeasible to \eqref{eq:project-out-x1}. Thus, the primal is not solvable. 
%

On the other hand, note that $(x_2, z) = (\delta, \frac{1}{\delta^2})$ is feasible to \eqref{eq:project-out-x1} whenever $\delta \ge 1$. 
Recursing back to system \eqref{eq:full-system-example} and letting $x_{2} = \delta$ and $z = \frac{1}{\delta^2}$ with $\delta \ge 2$ gives 
\begin{eqnarray*}
- x_{1} + \tfrac{1}{\delta^2} &\ge& 0 \\
\phantom{-} x_{1} + \tfrac{\delta}{t^{2}} &\ge& \tfrac{1}{t^{2}} + \tfrac{1}{t}, \quad t \ge 1 \\
x_1 &\ge& 0
\end{eqnarray*}
setting $t = \delta$ yields 
\begin{eqnarray*}
x_1 &\le& \tfrac{1}{\delta^2} \\
x_1 &\ge& \tfrac{1}{\delta^2},
\end{eqnarray*}
implying $x_1 = \frac{1}{\delta^2}$. Therefore for $\delta \ge 1$, $(x_1, x_2) = (\frac{1}{\delta^2},\delta)$ is feasible to \eqref{eq:not-primal-optimal} with objective value $\frac{1}{\delta^2}$. Letting $\delta \to \infty$ gives a sequence of feasible solutions whose objective values converge to the optimal value of $0$, but as observed above, there is no feasible solution that attains this value.}

We claim that $(x_1, x_2) = (\frac{1}{\delta^2}, \delta)$ is feasible to \eqref{eq:not-primal-optimal} whenever $\delta \ge 1$. 
Note $\frac{\delta}{t^2} \le \frac{t}{t^2} = \frac{1}{t}$ whenever $t \ge \delta$ and so $\frac{1}{t^2} + \frac{1}{t}- \frac{\delta}{t^2} \le \frac{1}{t^2} \le \frac{1}{\delta^2}$ whenever $t \ge \delta$. Thus, $\frac{1}{t^2} + \frac{1}{t} \le \frac{1}{\delta^2}+\frac{\delta}{t^2}$ for all $t \geq \delta$. On the other hand, when $1 \le t < \delta$ note that $\frac{1}{t} > \frac{1}{\delta}$ and $\frac{1}{t^2} > \frac{1}{\delta^2}$ and so $\frac{1}{t^2} + \frac{1}{t} - \frac{\delta}{t^2} \le \frac{1}{\delta^2} + \frac{1}{\delta} - \frac{\delta}{\delta^2} = \frac{1}{\delta^2}$. Thus, $\frac{1}{t^2} + \frac{1}{t} \le \frac{1}{\delta^2}+\frac{\delta}{t^2}$ for all $1 \le t < \delta$. Letting $\delta \to \infty$ gives a sequence of feasible solutions whose objective values converge to the optimal value of $0$. Since $x_1 \geq 0$ is a constraint, the optimal primal objective value is thus equal to $0$. 


It is easy to see that the finite support dual for this semi-infinite linear program is infeasible: The objective coefficient of $x_1$ in the primal is $0$ and the coefficient of $x_1$ is strictly positive in the constraints, and so the only possible dual element satisfying the dual constraint correpsonding to $x_1$ is $u = 0$; however, the objective coefficient of $x_2$ is $1$ and thus, we cannot satisfy the dual constraint corresponding to $x_2$. This shows that using the finite support dual leads to an infinite duality gap. On the other hand, if we view the constraint space as $c$ (or $\ell_\infty$) instead of $\R^\N$, then the limit functional defined in Lemma~\ref{lem:limit-func} (respectively, its extension to $\ell_\infty$ by Lemma~\ref{lem:limit-func-extend}) gives an optimal dual solution with optimal value $0$, giving zero duality gap. \hfill$\triangleleft$
\end{example}
}






\section*{Acknowledgements}

\noindent The authors  are very grateful to Greg Kuperberg, University of California Davis,  for discussions that led to the proof of Lemma~\ref{lemma:pos-functionals}. We are also thankful for the useful suggestions of our anonymous reviewer. 

\bibliographystyle{plain}
\bibliography{../../references/references}

\appendix

\section{Details on Remark~\ref{remark:algebraic-dual-big-cone-small}}\label{ap:no-interior}

\noindent This short appendix contains ancillary material pertaining to Remark~\ref{remark:algebraic-dual-big-cone-small} and the unsubstantiated claim that $\R^\N$ has no interior points in any linear topology. This claim follows from Corollary~9.41 in \cite{hitchhiker} and the fact $\R^\N$ has no order unit. An \text{order unit} $e$ in a Riesz space $X$ is a positive element where for every vector $x \in X$ there exists a $\lambda > 0$ such that $|x| \le \lambda e$. See page 322 of \cite{hitchhiker} for more details.  To see that $\R^\N$ has no order unit, consider candidate vector $e = (e_1,e_2, \dots)$.  There are two cases to consider. Case 1:  the components $e$  are bounded. Then there exists an $M >1$ such that $e_i \le M$ for all $i$. Take $x = (M, M^2, \dots, M^n, \dots)$. Then  $e$ is not an order unit since, for sufficiently large $n$, there does not exist $\lambda > 0$ such that $M^n \le \lambda e_i$ for all $i$.
Case 2: the components of $e$ are unbounded.  Then $e$  contains a subsequence $e_{i_k}$ where $e_{i_k} \to \infty$.  Now consider $x$ where $x_i = e_i^2$ for $i = 1, 2, \dots$. Setting $\lambda \ge e_{i_k}^2/e_{i_k} = e_{i_k}$ for all $k$ is impossible since $e_{i_k} \to \infty$. Thus, $e$ is not an order unit since $e_{i_k} \to \infty$.







\end{document}